\documentclass[12pt]{amsart}
\usepackage{amsmath, yhmath}
\usepackage{comment}
\usepackage{enumerate}
\usepackage{hyperref}
\usepackage{breqn,cleveref}
\usepackage{amssymb}
\usepackage{xcolor}
\newcommand{\bi}{\bar{i}}
\newcommand{\bj}{\bar{j}}
\newcommand{\bk}{\bar{k}}
\newcommand{\bl}{\bar{l}}

\newcommand{\bp}{\bar{p}}

\newcommand{\bs}{\bar{s}}

\newcommand{\bw}{\bar{w}}

\newcommand{\bbC}{\mathbb{ C}}

\newcommand{\bbR}{\mathbb{ R}}

\newcommand{\p}{\phi}

\newcommand{\bpartial}{\bar{\partial}}

\newcommand{\pbp}{\partial \bar{\partial}}

\newcommand{\red}[1]{\textcolor{red}{#1}}

\newcommand{\al}{\alpha}
\newcommand{\be}{\beta}
\newcommand{\ga}{\gamma}
\newcommand{\de}{\delta}

\newcommand{\oal}{\bar{\alpha}}
\newcommand{\obe}{\bar{\beta}}

\newcommand{\ode}{\bar{\delta}}

\newcommand{\fRe}{\mathfrak{Re}}

\newcommand{\cC}{\mathcal{C}}

\newcommand{\ol}{\overline}
\newcommand{\ul}{\underline}
\theoremstyle{definition}
\newtheorem{Def}{Definition}[section]
\newtheorem{ex}{Example}[section]
\newtheorem{theorem}{Theorem}[section]
\newtheorem{lemma}[theorem]{Lemma}
\newtheorem{prop}[theorem]{Proposition}
\newtheorem{cor}[theorem]{Corollary}

 \theoremstyle{definition}

\theoremstyle{remark}
\newtheorem{Rmk}[theorem]{Remark}

\numberwithin{equation}{section}

\begin{document}

\title[Volume forms on balanced manifolds]
{Volume forms on Balanced manifolds and the Calabi-Yau Equation}
%on Hermitian manifolds}
\author{Mathew George}
\address{Department of Mathematics, Purdue University,
        West Lafayette, IN 47907, USA}
\email{georg233@purdue.edu}
%\author{Bo Guan}
%\address{Department of Mathematics, Ohio State University,
%         Columbus, OH 43210, USA}
%
%\email{guan@math.ohio-state.edu}

\date{}

\begin{abstract}

 We introduce the space of mixed-volume forms endowed with a $L^2$ metric on a balanced manifold. A geodesic equation can be derived in this space that has an interesting structure and extends the equation of Donaldson \cite{Donaldson10} and Chen-He \cite{CH11} in the space of volume forms on a Riemannian manifold. This nonlinear PDE is studied in detail and we prove several estimates, under a positivity assumption. Later we study the Calabi-Yau equation for balanced metrics and introduce a geometric criterion for prescribing volume forms, that is closely related to the positivity assumption above. By deriving $C^0$ a priori estimates, we prove the existence of solutions on all such manifolds.

%{\em Keywords:} Fully nonlinear elliptic equation, %Hermitian manifolds,
%{\em a priori} estimates, Dirichlet problem, subsolutions,  %concavity and convexity properties.
\end{abstract}

\maketitle

\vspace{2cm}
\section{Introduction}

Given a Hermitian manifold $(M, \omega)$, we say that $\omega$ is balanced if $d\omega^{n-1} =0$. This is equivalent to requiring that the trace of the torsion endomorphism of $\omega$ vanishes identically. These metrics were introduced by Michelsohn \cite{Michelsohn82} in 1982 as an alternative to K\"ahler metrics, which are known to impose many topological and geometric restrictions on a complex manifold. Balanced metrics can be seen as dual to K\"ahler metrics in a sense made precise by Michelsohn \cite{Michelsohn82}. Recently, they have gained relevance because of their applications in string theory, and in birational geometry. For example, the Strominger system \cite{Strominger86} consists of a system of coupled nonlinear equations on a complex $3$-fold $X$ and a bundle $E \to X$ over it, parts of which have been simplified by Li and Yau \cite{LY2005} to the problem of finding a conformally balanced metric. This can be reduced to a Calabi-Yau-type equation for balanced metrics, which will be discussed in Section \ref{CYB}. In birational geometry, balanced metrics are important, as the existence of balanced metrics is preserved under birational transformations \cite{AB96}. Hence it is thought that balanced metrics might give an important class of canonical metrics in non-K\"ahler geometry. For more details, we refer to \cite{Fu16, FLY12, TW10} and references therein.

In this paper, we consider the space of mixed-volume forms on a balanced manifold. A geodesic equation is derived in this space which yields a new nonlinear PDE which we wish to study in-depth. We find an interesting positivity assumption coming from the study of this equation which is also related to the problem of prescribing volume forms for balanced metrics that can be written as an $(n-1)$ Monge-Amp\`ere equation, similar to the Gauduchon conjecture \cite{STW17}.

The space of K\"ahler metrics on a K\"ahler manifold with an $L^2$ metric structure has been studied extensively starting with Mabuchi \cite{Mabuchi87}, Donaldson \cite{Donaldson99}, Semmes \cite{Semmes92}, and later by Chen \cite{Chen2000} and many others. Similar structures have also been introduced in the space of volume forms on a Riemannian manifold by Donaldson \cite{Donaldson10}. Such spaces seem to have interesting properties. For example, the geodesic equation in the space of K\"ahler potentials can be transformed into a degenerate complex Monge-Amp\`ere equation in one dimension higher. These find applications in geometric problems such as the uniqueness of constant scalar curvature metrics in a K\"ahler class when $c_1(M) \leq 0$. In the case of the space of K\"ahler metrics, geodesic rays are related to the Yau-Tian-Donaldson conjecture on the existence of cscK metrics.

These equations are generally degenerate and involves finding a weak solution to the geodesic equation corresponding to the given metric. It is of interest to extend such structures to Hermitian geometry. In the K\"ahler case, there are many simplifications especially in the variational computations that makes it possible to study these structures. Although this does not seem to be true in general, the balanced property might be sufficient in some cases.

% Here is an example of a balanced manifold which does not admit a K\"ahler metric \cite{}.

% \begin{ex} 

% {\bf Iwasawa manifold:} $$M:=\{A\in GL_3(\mathbb{C})|A=\begin{bmatrix}1 & z_1 &z_3\\
%     0 & 1 &z_2\\
%     0 & 0 &1
%     \end{bmatrix}, z_i\in \mathbb{C})\}$$

% \noindent Consider the action on $M$ by the subgroup $G\subset M$ of matrices with coefficients being Gaussian integers, given by left multiplication. The Iwasawa threefold $X$ is defined as the quotient of $M$ under this action. The holomorphic one forms $dz_1, dz_2$ and $dz_3-z_1dz_2$ are invariant under this action.  A balanced metric can be explicitly defined on $X$ as follows.

% $$idz_1\wedge d\bar{z}_1+idz_2\wedge d\bar{z}_2+i(dz_3-z_1dz_2)\wedge \overline{(dz_3-z_1dz_2)}$$

% \noindent As $d(dz_3-z_1dz_2)=-dz_1\wedge dz_2$ is also invariant, it descends to a non-zero 2-form. Thus $dz_3-z_1dz_2$ is a non-closed holomorphic one form on $M$. This shows that $X$ does not admit K\"ahler metrics (by Hodge theorem all holomorphic 1-forms on a K\"ahler manifold are harmonic and hence closed). So the space of balanced manifolds is strictly bigger than the space of K\"ahler manifolds.

% \end{ex}

%We refer to \cite{Fu16} for a survey on balanced manifolds. More details and examples from non-K\"ahler geometry can be found in \cite{DPTV19}.

Let $(M,\omega)$ be an $n$-dimensional closed balanced manifold. That is, $ (M, \omega )$ is a Hermitian manifold with the metric $\omega$ satisfying $d\omega^{n-1}=0$. Then for any smooth function $\phi$ on $M$, define a $ ( p, p ) $ form by 

$$\Omega_{\phi}=\omega^{p}+\sqrt{-1}\pbp(\phi \omega^{p-1}).$$

If $\omega^p $ is closed, then these forms are in the same $p^{th}$ Bott-Chern cohomology class $H_{BC}^{p,p}(M, \mathbb R)$. We consider the space of mixed-volume forms of order $p$ parametrized by smooth functions on $M$ in the following way.

\begin{equation}\label{I0.1}
    \mathcal{V}_p = \{ \phi\in C^{\infty}(M) : \Omega_{\phi} \wedge \omega^{n-p} >0 \}
\end{equation} 

 Then $\mathcal{V}_p$ is an infinite dimensional manifold with tangent space at any point identified with the set of all smooth functions on $M$.

$$T_\phi{\mathcal{V}_p} \cong \{\psi\in C^{\infty}(M)\}$$

\noindent The space $\mathcal{V}_p$ is endowed with the following $L^2$ metric.

\begin{equation}\label{I1}
    ( \psi_1, \psi_2 )_{\phi} = \left( \int_M \psi_1 \psi_2 \; \Omega_{\phi} \wedge \omega^{n-p} \right)^{\frac{1}{2}}
\end{equation}

The geodesic equation in $\mathcal{V}_p$ with respect to this metric is given by

\begin{equation} \label{I1.1}
\begin{aligned}
    \phi_{ t t } ( n + n X \phi + \Delta \phi) - | \nabla \phi_t |^2 = - \frac{ n X \phi_t^2 }{2},\\
    \end{aligned}
\end{equation}

\noindent with boundary conditions

$$ \phi(x,0) = \phi_0, \hspace{1cm} \p(x,1) =\p_1,$$

\noindent for a non-negative function $X$ involving $p$  and the torsion tensor of $( M, \omega )$ (see Section \ref{GE}).

The case that is particularly interesting is when $ p = n-1 $ so that  $\Omega_{ \p } = \omega_{\p}^{n-1} $ defines a $(1, 1)$ form which is also a balanced metric when $\omega_{\p}>0$. This cohomology relation is important, for instance in Calabi-Yau type theorems for balanced metrics, where we search for a balanced metric $\omega_{\phi}$ with a prescribed volume form. In this case, the ellipticity cone is contained in $\mathcal{V}_{n-1}$ defined above.

\begin{comment}
  Such geometric structures were first introduced by Mabuchi \cite{} in the space of K\"ahler potentials, which were later studied in more  detail by \cite{} \cite{} \cite{}. Later,  
\end{comment} 

The space of volume forms on a Riemannian manifold was initially introduced by Donaldson in the context of a free boundary problem related to Nahm's equation \cite{Donaldson10}. This is given by 

\begin{equation} \label{I2}
    \mathcal{V}=\{\phi\in C^{\infty}(M) : 1 - \Delta \phi > 0\}
\end{equation}

\noindent with the metric on $T_{\phi}M$, 
 
 \begin{equation}\label{I2.1}
     ||\psi||^2=\int_{M}\psi^2(1-\Delta \phi)dV.
 \end{equation}

The geodesic equation in this case is

\begin{equation} \label{I2}
    \p_{t t }( 1 - \Delta \p) - \sum_k \p_{t k}^2 =0.
\end{equation}

This is sometimes also referred to as the Donaldson equation and was shown to have $C^{1, \alpha }$ weak solutions by Chen-He \cite{CH11}. The regularity was subsequently improved to $C^{1 , 1}$ by Chu \cite{Chu19}. There have been subsequent works by Chen-He \cite{CH19} and He \cite{He21} extending this equation to cover, in particular, certain cases of the Streets-Gursky equation \cite{GS19}. In the case when $M$ is K\"ahler, equation \eqref{I1.1} will be identical to \eqref{I2}, since $X=0$ for K\"ahler manifolds.

We aim to study the equation \eqref{I1.1} in detail. It is clear that the sign of $X$ is an important factor for this equation. In this paper we assume that $X\leq 0$, so that the equation is degenerate elliptic. In later works, we hope to consider the geometric case when $X \geq 0$ so that the equation is degenerate hyperbolic.

The techniques from \cite{CH11} cannot be applied to equation \eqref{I1.1}. The major obstacle in deriving estimates are the terms involving the function $X$. The structure of the equation is such that there are some important cancellations that enable us to overcome this.

For avoiding degeneracies, we consider the perturbed equation 

\begin{equation} \label{pde}
\begin{cases}
    \p_{ t t }( n + n X \p + \Delta \p) - |\nabla \p_t|^2 = \epsilon - \dfrac{ n X \p_t^2 }{2} \\
     {\p}( x, 0 ) = \phi_0, \hspace{1cm} \\
     {\p}( x, 1 ) = \phi_1.
\end{cases}
\end{equation}

\noindent and then take limits as $\epsilon \to 0$. Here $\p_0$ and $\p_1$ are assumed to be smooth. A subsolution $\underline \p$ is a smooth function satisfying the following.
\begin{equation}\label{subsolution}
        \underline \p_{ t t }( n + n X \underline \p + \Delta \underline \p) - |\nabla \underline \p_t|^2 > \epsilon - \frac{ n X \underline \p_t^2 }{2}
\end{equation}
\noindent and the boundary conditions
\begin{equation}\label{sub-bound}
    \begin{aligned}
 \ul{\p}( x, 0 ) = \phi_0, \hspace{1cm} \text{and} \hspace{1cm} \ul{\p}( x, 1 ) = \phi_1.
    \end{aligned}
\end{equation}

Denote $Y = M \times [0,1]$. The following estimates will be shown in this paper.

\begin{theorem}\label{balanced-main}
Let $\p\in C^4(Y)$ be solution of \eqref{pde}. Assume that a subsolution $\underline \p$ satisfying equation \eqref{subsolution} and \eqref{sub-bound} exists and $X \leq 0$. Then we have the following estimates
\begin{equation}\label{theorem-geo}
\begin{aligned}
    &\sup_{Y} ~| \p_{tt} | \;\leq C\\ 
    &\sup_{\partial Y} ~ (|\p_{tt}| +  |\nabla \p_{t} |) \; \leq C\\   
\end{aligned}
\end{equation}
\end{theorem}

\noindent for a constant $C$ that depends only on $(M,\omega)$, $\underline{\p}$ and other known data.

However, to show the existence of weak solutions, all estimates up to second-order are required, which we pose as an open question.

\

The second aim of this paper is to study a Calabi-Yau-type theorem for balanced metrics. We state the main statement here and rest of the details are presented in Section \ref{CYB}. 

\begin{theorem}\label{theoremB}
Let $(M, \omega)$ be a balanced manifold such that there exists a Hermitian metric $\alpha$ on $M$ with

\begin{equation}\label{alpha+}
    \partial \bpartial \alpha^{n-2} \leq 0
\end{equation}
    \noindent as an $(n-1, n-1)$ form. Then given a $(1,1)$ form $\Psi$ in $H^{1,1}_{BC}(M, \mathbb R)$, there exists a balanced metric $\omega'$ such that $[{\omega'}^{n-1}] = [\omega^{n-1}] $ in $H_{BC}^{n-1,n-1}(M, \mathbb R)$, and

    \begin{equation}\label{chern-ricci}
        \mathrm{Ric}^{C}(\omega') = \Psi.
    \end{equation}
    
\end{theorem}

Here $\mathrm{Ric}^C(\omega') = - \sqrt{-1} \pbp \log{{\omega'}^n}$ is the Chern-Ricci form associated to the metric $\omega'$. This is also equivalent to prescribing a volume form for the metric $\omega'$. Note that the assumption $X\leq 0$ in Theorem \ref{balanced-main} can be obtained by taking the $\alpha$-trace of \eqref{alpha+}.

It can be shown that equation \eqref{chern-ricci} can be transformed into

\begin{equation}\label{CYBeqn}
    \det \left(\omega_h + \frac{1}{(n-1)}\left(\Delta u ~\alpha - \sqrt{-1} \pbp u\right) + \chi(\partial u, \bpartial u) + E u \right) = e^{\psi+b} \det \alpha,
\end{equation}

\noindent with

$$\tilde{\omega}_u = \omega_h + \frac{1}{(n-1)}\left(\Delta u \alpha - \sqrt{-1} \pbp u\right) + \chi(\partial u, \bpartial u) + E u > 0.$$

\noindent  for an unknown function $u$ and a constant $b$. Refer to Section \ref{CYB} for the definitions of the terms involved. This equation has been observed in \cite{STW17}, where the authors solve it assuming that $E=0$, and with additional symmetry assumptions on $\chi(\partial u, \bpartial u)$. When $E\neq 0$, there are many difficulties, mostly caused by the fact that the maximum principle does not apply for many of the arguments. This should be compared to the case of linear equations when the coefficient of the zeroth-order term does not have a good sign.  

 See \cite{FWW10,FWW15,TW17,TW19, GGQ22, GN23, STW17} for the theory of equations involving $(n-1)$ plurisubharmonic forms and \cite{TW10} for the complex Monge-Amp\`ere equation on balanced manifolds. Theorem \ref{theoremB} is obtained as consequence of the following.

\begin{theorem}
    Assuming that there exists a Hermitian metric $\alpha$ satisfying \eqref{alpha+}, there exists a unique constant $b$ and a unique smooth function $u$ that solves the equation \eqref{CYBeqn}.
\end{theorem}

We give two different proofs of this theorem. The second method is more general and will only use that $\mathrm{tr}_{\tilde{\omega}_u} E \leq 0$. This suggests a potential strategy for solving the problem in general, by considering a continuity path along this direction. The assumption $\sqrt{-1} \pbp \omega^{n-2} \leq 0$ seems to be interesting from a geometric perspective. These are discussed in Section \ref{FR}.

\begin{comment}
   Since all the functions in $ \mathcal{ V }_p $ are uniformly bounded above,

\begin{equation}
      \phi < C_p \; \forall \; \phi \in \mathcal{V}_p
\end{equation}

\noindent for a constant $C_p$ that depends only on $ (M, \omega) $, $ n $ and $ p $.

Observe the following relationship between the space of mixed volume forms for different $p$ values that
follow directly from the defining equations for $ \mathcal{ V }_p $ and using that $X$ is non-negative.

\begin{enumerate}

\item
     $\mathcal{ \tilde{ V } }_p = \mathcal{ \tilde{ V } }_{n - p +1} $ for all $p$.
    
    \item There are natural inclusions,
    
    $$\mathcal{ V }_1 \hookleftarrow \mathcal{  V  }_2 \hookleftarrow \hdots \hookleftarrow \mathcal{  V  }_{ \lfloor \frac{n+1}{2} \rfloor }$$
    
\end{enumerate}

\red{actual inclusions and add this as a comment}

\end{comment}

In the Section \ref{GE}, the geodesic equation is derived and various properties associated to the balanced condition are shown. In the later sections, we derive interior estimate for $|\p_{tt}|$ and boundary $C^2$ estimates for the solution. In Section \ref{CYB}, we show $C^{0}$ estimates for the balanced Calabi-Yau equation.

\

{\bf Acknowledgements:} I would like to thank Professor Ben Weinkove, Professor Song Sun, Professor Xiangwen Zhang, Nicholas McCleerey, and Bin Guo for helpful discussions.

\

\section{The Geodesic equation}

\label{GE}

Throughout this article, derivatives in the time variable will always be denoted by subscripts in $t$, so that $\nabla \p$, $\Delta \p$ denote only the space derivatives given by the Chern connection of $M$. We begin by deriving several identities satisfied by a balanced metric. 

\begin{lemma}\label{B2}
Assume $d \omega^{n-1}=0$. Then the following are true for any $2 \leq p \leq n$.
\begin{enumerate}[(i)]
    % \item  $\partial \omega^{n-2}\wedge \omega=-\omega^{n-2}\wedge\partial \omega =0$.
    % \red{redo this to make it look good}
    % \item $ \pbp \omega^{n-2}\wedge \omega=(n-2)\; \bpartial\omega\wedge \partial \omega \wedge \omega^{n-3}$.

    \item \label{B2.1} $\partial \omega^{p-1} \wedge \omega^{n-p} = 0 $

    \item \label{B2.2}$ \pbp \omega^{p-1} \wedge \omega^{n-p} = ( n - p ) (p - 1) \bpartial \omega \wedge \partial \omega \wedge \omega^{n-3}$

    \item \label{B2.3} Define a function $X$ by
    
    $$ X \omega^n = \sqrt{-1} \pbp \omega^{p - 1}\wedge \omega^{n - p}.$$

    Then $X \geq 0$.  
\end{enumerate}
\end{lemma}

\begin{proof} First two parts are applications of the Leibniz rule.

$$ \partial \omega^{p-1} \wedge \omega^{n-p} = (p - 1) \omega^{p-2} \wedge \omega^{n - p} \wedge \partial \omega = \frac{p-1}{n-1}\partial \omega^{n-1} = 0$$

To get \eqref{B2.2}, we compute

$$ \pbp \omega^{p-1} \wedge \omega^{n-p} = (p - 1)(p - 2) \partial \omega \wedge \bpartial \omega \wedge \omega^{n-3} + (p-1) \pbp \omega\wedge \omega^{n-2}$$

\

Applying $\bpartial$ to \eqref{B2.1} with $p= 2$ gives

$$\pbp \omega \wedge \omega^{n-2} = -(n-2)\; \partial \omega \wedge \bpartial \omega \wedge \omega^{n-3} .$$

\eqref{B2.2} now follows by combining the above two equations.
For showing \eqref{B2.3},
we compute in orthonormal coordinates at a point. Following the convention in \cite{Michelsohn82},

$$\partial \omega = \sqrt{-1} T_{j k }^l dz_j \wedge dz_k \wedge dz_{\bl} $$

$$ \bpartial\omega = - \sqrt{-1} \overline{T_{ i p}^q} dz_{\bi} \wedge dz_{\bp} \wedge dz_{q}$$

\begin{equation}
    \begin{aligned}
        \sqrt{-1}\bpartial \omega \wedge &\partial \omega \wedge \omega^{n-3} =  (\sqrt{-1})^{n} (n-3)!( T_{j k }^l dz_j \wedge dz_k \wedge dz_{\bl})\wedge (\overline{T_{ i p}^q} dz_{\bi} \wedge dz_{\bp} \wedge dz_{q})\wedge \\
        &\left(\sum\limits_{a < b< c}dz_1 \wedge dz_{\bar 1} \wedge \hdots \widehat{dz_{a} \wedge dz_{\overline a}} \hdots \widehat{ \wedge dz_{b} \wedge dz_{\overline b} } \wedge \hdots \wedge \widehat{ dz_{c} \wedge dz_{\overline c} } \wedge \hdots dz_n \wedge dz_{\bar n} \right)\\
        &=  \frac{2}{n (n-1)(n-2)} ( \sum\limits_{j, k, l, j \neq l}|T_{j k }^l|^2 - \sum_{j,k,l}  \left( T_{j k }^j \overline{T_{lk}^l} + T_{jk}^k  \overline{T_{j l}^l} \right)) \omega^n
    \end{aligned}
\end{equation}

Here the $2$ in numerator comes from the anti-symmetry of $T_{ij}^k$ in indices $i$ and $j$. By using that $\sum_iT_{ij}^i=\sum_jT_{ij}^j=0$ for balanced metrics, the second and third terms in the above expression vanishes.

It follows from above and using \eqref{B2.2} that at a point where $g_{i\bj}=\delta_{ij}$,

\begin{equation}\label{B7}
    \begin{aligned}
      \sqrt{-1} \pbp \omega^{p - 1}\wedge \omega^{n - p}= \frac{2(n - p)(p - 1)}{n( n - 1 )(n - 2)}\sum\limits_{i\neq k}|T_{ij}^k|^2\omega^n
    \end{aligned}
\end{equation}

\noindent where $T_{ij}^k$ denote the components of the torsion tensor. This shows that $X \geq 0$.
\end{proof}

\begin{Rmk}
   From \eqref{B7}, we can also make the following observation.

   $$\mathcal{ V }_p = \mathcal{ V }_{n - p +1} $$ 
   \noindent for all $p$. This holds since the expression for $X$ is symmetric with respect to this transformation. Also see \eqref{B3}. 
\end{Rmk}

We now derive the equation of a geodesic segment joining $\phi_0$ to $\phi_1$ in $\mathcal{V}_p$ by minimizing the following energy functional.

\begin{equation}\label{G1}
    \begin{aligned}
    \mathcal{E}^2=\int_{0}^1|| \p_t ||^2dt=\int_{0}^1\int_M \p_t ^2 \;\Omega_{\phi} \wedge \omega^{n-p} dt
    \end{aligned}
\end{equation}

Let $\phi^s(t, .)$ be an end-point fixing variation of paths in $\mathcal{V}_p$ such that $\phi^s(.,0)=\phi_0(.)$ and $\phi^s(.,1)=\phi_1(.)$, with $s \in [-1, 1]$.

Using Lemma \ref{B2},

\begin{equation}\label{B3}
    \begin{aligned}
    \Omega_{\phi}\wedge \omega^{n-p} = \omega^n+\frac{1}{n}\Delta\phi\omega^n+ X \phi \omega^n.
    \end{aligned}
\end{equation}

Now the energy becomes

\begin{equation}\label{B4}
    \begin{aligned}
    \mathcal{E}^2=\int_0^T\int_M \p_t ^2\left(1+\frac{\Delta\phi}{n}\right)\omega^n+ \p_t ^2\phi X \omega^n dt
    \end{aligned}
\end{equation}
Assuming that $\phi^0 = \p$ minimizes $\mathcal{E}$, we have

\begin{equation}\label{G2}
    \begin{aligned}
    0&=\left.\frac{\partial }{\partial s}\mathcal{E}^2\; \right\vert_{s=0}\\
    &=\int_0^T\int_M 2 \p_t  \psi_t \left(1+\frac{\Delta\phi}{n}\right)+ \p_t ^2\frac{\Delta\psi}{n}\omega^n+\left(2 \p_t  \psi_t \phi+ \p_t ^2\psi\right) X \omega^n dt
    \end{aligned}
\end{equation}

\noindent where $\psi=\left.\frac{\partial }{\partial s}\phi\; \right\vert_{s=0}$ is the variational field. Performing standard variational calculus gives the following geodesic equation.

\begin{equation}\label{B5}
    \begin{aligned}
    & \p_{ t t } ( n +  n X \p + \Delta \phi ) -|\nabla \p_t |^2 + \frac{n X \p_t^2 }{2}  =0
    \end{aligned}
\end{equation}

\noindent with $\phi(.,0)=\phi_0$ and $\phi(.,1)=\phi_1$. An important point here is that integration by parts uses the balanced condition and hence this construction will not generalize easily to any Hermitian metric. From now on we will use the notation $\phi(z,t)$ to denote the geodesic segment joining $\phi_0$ and $\phi_1$.

\section{Preliminaries}

\label{C}

In this section, we will introduce some basic lemmas and the setup for continuity method. Assume $X\leq 0$ . Let 

$$A( \p ) = n + n X \p + \Delta \p,$$
      \noindent and
        
        $$G(\p) = \phi_{ t t } A(\p ) - \sum_k |\p_{ k t}|^2.$$

       Note that $G(\p)>0$ for a solution $\p$. We also denote  $L(\phi) = \epsilon - \dfrac{n X \p_t^2}{2} > 0.$

\

 Greek indices are used to denote both space and time variables whereas English indices are for space variables only. Denote

\begin{equation}
    \begin{aligned}
        F^{\al \obe} = \frac{\partial F}{ \partial \p_{\al \obe}},\hspace{1cm} F^{\al \obe, \ga \ode } = \frac{\partial^2 F}{\partial \p_{\al \obe } \partial \p_{ \ga \ode}}.
    \end{aligned}
\end{equation}

Consider the function $f : \mathbb{R}^{n+2} \to \mathbb{R} $ given by

$$f(x, y, z_1, z_2, \hdots, z_n ) = \log{ ( xy-\sum_k z_k^2 ) } $$

 It was proven in \cite{Donaldson10}, and later also in \cite{CH11} that 

 \begin{lemma}
    $ f(x, y, z_1, z_2, \hdots, z_n )$ is concave in the set where $x>0$, $y>0$, and  $xy-\sum_k z_k^2 >0 $.
 \end{lemma}

 We need an extension of this lemma to the complex case. That is, $- \log{ (xy-\sum_k z_k \bar z_k ) }$ is plurisubharmonic. This follows directly from the following proposition. See Theorem 5.6 in \cite{Demailly12}. 

 \begin{prop}
    Let $u_1, \hdots ,u_p$ be plurisubharmonic functions defined in a domain $ \Omega $ and $ \chi: \mathbb{R}^p \to \mathbb{R} $ be a convex function such that $\chi(t_1, \hdots , t_p)$ is non-decreasing in each $t_j$. Then $\chi(u_1, \hdots, u_p)$ is plurisubharmonic on $\Omega$.
 \end{prop}

 It follows from the above two results that the function $g : \mathbb{R}^2 \times \mathbb{C}^{n} \to \mathbb{R}$ given by

 \begin{equation}\label{plurisub}
    g(x, y, z_1, \hdots , z_n)= -\log( x y - \sum_k |z_k|^2 )
 \end{equation}

\noindent is  plurisubharmonic in $\bbC^n$ when $ x, y \in \mathbb R^+$ and $ x y  - \sum_k |z_k|^2>0 $.

\

Denote the nonlinear operator  

\begin{equation}\label{C1}
F( D^2 \phi, \phi, z) = \phi_{tt} ( n + n X \phi + \Delta \phi ) - | \nabla \phi_t |^2. 
\end{equation}

The continuity path is given by

\begin{equation}\label{C2}
   P_s(  D^2 \phi,  \phi, z ) = s F( D^2 \phi, \phi, z ) + ( 1 - s ) ( \phi_{tt} + A( \phi) ) = \epsilon - \frac{nsX}{2} \p_t^2.
\end{equation}

To show existence, it is enough to show that there is a unique smooth solution for the Dirichlet problem

\begin{equation} \label{C3}
\begin{aligned}
      & P_s(  D^2 \phi, D \phi, z ) = \epsilon - \frac{nsX}{2} \p_t^2 \\
      & \phi (. , 0) = \phi_0, \; \phi(. , 1) = \phi_1 \\
\end{aligned}
\end{equation}

\noindent for each $s \in [0,1] $. Let $S = \{ s \in [0,1] | \text{ \eqref{C3} has a unique smooth solution for } [0,s) \} $. Clearly $0 \in S$ and by implicit function theorem there is a $\delta >0 $ such that $[ 0, \delta) \subset S$. For showing that $1 \in S$ and hence the equation \eqref{C3} has a smooth solution, we need to derive a priori estimates up to boundary for \eqref{C3}.

For simplicity, consider the equation at $s=1$. That is, the equation

\begin{equation}\label{C3.1}
    F( D^2 \p, \p, z) = L.
\end{equation}

The calculations for general $s$ are similar. The linear operator associated to $F$ at some $\phi$ is given by

\begin{equation} \label{C4}
    \mathcal{L} u =  A( \phi ) u_{tt} + \phi_{tt} \Delta u - 2 \fRe{(\p_{t \bk} u_{t k})} 
\end{equation}

 It follows that the principal symbol can be written as the following $(n+1) \times (n+1)$ matrix.

\begin{equation}\label{C5}
\begin{bmatrix}
     A(\p) & - \nabla_1 \phi_t & \cdots & -\nabla_n\p_t \\
     - \nabla_{ \bar{1} } \phi_t &  \p_{t t} & 0 & 0\\
     \vdots & 0 & \ddots & 0\\
     - \nabla_{ \bar{n} } \phi_t & 0 & 0&  \p_{t t}
\end{bmatrix}
\end{equation}

We prove some basic results that will be useful later.

\begin{lemma} \label{ellipticity}
    $ F( D^2 \p, \p, z ) $ is elliptic at a solution $\p$ of \eqref{C3.1}.
\end{lemma}

\begin{proof}
    We show this by proving that the matrix \eqref{C5} is positive-definite.

From \eqref{C3.1}, 

\begin{equation}\label{C'5.1}
    \sum |\p_{ t k }|^2 < \p_{ t t } A(\p)
\end{equation}

Given any vector $\xi \in \mathbb{C}^n$, we can compute

\begin{equation}
    \begin{aligned}
        F^{\al \obe } \xi_{\al} \ol{ \xi_{\be} } = A(\p) |\xi_t|^2 + \p_{ t t } \sum_k |\xi_k|^2 - \sum_{k} (\p_{ k t } \xi_{\bk } \xi_{t} + \p_{ \bk t } \xi_{ k } \xi_{t} )
    \end{aligned}
\end{equation}

From \eqref{C'5.1},
\begin{equation}
\begin{aligned}
    \sum_k \p_{ k t } \xi_{\bk } \xi_{t} & \leq \sqrt{ \sum_k |\p_{ k t }|^2 } \sqrt{\sum_k |\xi_k |^2 } \xi_t\\
    & <  \frac{1}{2} \left(\p_{ t t } \sum_k |\xi_k |^2 + A(\p)| \xi_t |^2 \right)
\end{aligned}
\end{equation}

It follows that $F^{\al \obe } \xi_{\al} \ol{ \xi_{\be} } > 0 .$
    
\end{proof}

\begin{lemma}\label{concavity}
    Let $F: \mathbb{S}^{ n+1} \to \mathbb{R}$ be a function defined on the space of symmetric matrices as follows
\begin{equation}\label{C5.1}
    F(A) = A^{00} \sum\limits_{i=1}^n A^{ i i } - \sum\limits_{i=1}^n ( A^{ i 0 } )^2.
\end{equation}

Then 
    \begin{enumerate}
        \item \label{concavity1} $F$ is concave.
        
        \item \label{concavity2}  For all $B$ such that $F(B) > F(A)$,
        
        \begin{equation} \label{C5.2}
            \sum F^{ i \bj } ( B_{i \bj} - A_{i \bj} ) \geq \epsilon \sum F^{i \bi}
        \end{equation}
        
        \noindent for some small $\epsilon>0$.
    \end{enumerate}

\end{lemma}

\begin{proof}
    For part one refer to \cite{Donaldson10}.
Part two is a special case of Theorem $2.17$ from \cite{Guan14}. We give a simpler proof here.

Define 
$$ \Gamma^{ \sigma } = \{ A: F(A) > \sigma \} $$

\begin{comment}
and

$$\partial \Gamma^{ \sigma } = \{ A : F(A) =\sigma \} .$$
\end{comment}

Then since $B \in \Gamma^{F(A)}$, there exists an $\epsilon > 0$ such that $ B - \epsilon I \in \Gamma^{F(A)}$. By concavity of $F$,

\begin{equation} \label{C5.3}
    F^{ i \bj }( B_{i \bj} - \epsilon \delta_{i j} - A_{i \bj} ) \geq F(B - \epsilon I) - F(A) >0.
\end{equation}

Hence \eqref{C5.2} follows.

\end{proof}

As a consequence of \eqref{C5.2},

\begin{equation} \label{C5.4}
\begin{aligned}
    \mathcal{L} ( \ul{\p} - \p ) &\geq \epsilon_1 \sum F^{ \al \oal}\\
    &= \epsilon_1 ( n + n X \p + \Delta \p + n \p_{ t t }),
    \end{aligned}
\end{equation}

\noindent for some positive constant $\epsilon_1$.

\section{ $C^0$ and $\p_{t}$ Estimates}

\label{G}

Assuming the existence of a subsolution, we show that any solution of \eqref{C3.1} is bounded. By maximum principle, it is clear that they are bounded above.

\begin{prop}
    A $C^2$ solution $\p$ to  \eqref{C3.1} satisfies

    \begin{equation} \label{C6}
        \ul{ \phi } \leq \phi \leq \ol{ \phi }
    \end{equation}

    \noindent for some smooth bounded function $ \ol{ \phi } $.
\end{prop}

\begin{proof}

Let $ \ol{ \phi } $ be a solution to the Dirichlet problem

\begin{equation}
\begin{cases}
    & n + u_{ t t } + \Delta u + n X u = 0 \hspace{1cm} \text{ in } M \times ( 0, 1 ) \\
    & u( x, 0 ) = \phi_0 \\
    & u( x, 1 ) = \phi_1 \\
\end{cases}
\end{equation}

Then since $ \phi $ is a subsolution of this equation, it follows from the maximum principle that $ \phi \leq \ol{ \phi } $.

For the lower bound, assume for contradiction that $ \phi < \ul{ \phi } $ somewhere in the interior, so that $ \ul{\phi} - { \phi } $ attains a positive maximum at an interior point $q$.

From the subsolution, at the point $q$, we know that 
\begin{equation}\label{C7}
    F( \ul{ \phi } ) - F(  \phi  ) > 0
\end{equation}

So by concavity of $F$

\begin{equation}\label{C8}
    \mathcal{ L } ( \ul{ \phi } - \phi ) > 0 
\end{equation}

But by Lemma \ref{ellipticity}, $ \mathcal{ L } ( \ul{ \phi } - \phi ) \leq 0$ which contradiction.

\end{proof}

\

Boundary and interior estimates for $| \phi_t |$ can be shown as follows.

\begin{prop}
    For any solution $\phi$ of \eqref{C3.1}, there is a uniform constant $C$ so that 

\begin{equation}
    \begin{aligned}
     \sup\limits_{ Y } |\phi_t |  \leq C \\
    \end{aligned}
\end{equation}

\end{prop}

\begin{proof}
Since $ \phi_{t t} \geq 0 $, integrating $\phi_{tt}$ in $ [0, t] $ and $[ t, 1 ]$ gives

\begin{equation}
\begin{aligned}
    &\phi_{ t } (t, z) \geq \phi_t ( 0 , z ), \hspace{1cm} \text{and} \hspace{1cm} \phi_{ t } (t, z) \leq \phi_t ( 1 , z ).
    \end{aligned}
\end{equation}

So it is enough to estimate $ \phi_t $ on the boundary. Observe that

\begin{equation}
\lim\limits_{ t \to 0^+} \frac{ \ul{\phi}( t, z ) - \ul{ \phi } ( 0, z ) }{ t } \leq \phi_t ( 0, z ) \leq \lim\limits_{ t \to 0^+}  \frac{ \ol{\phi}( t, z ) -  \ol{ \phi } ( 0, z ) }{ t }
\end{equation}

This shows that $| \phi_t( 0, z ) | \leq C $. Similarly, one can show that $| \phi_t( 1, z ) | \leq C $.

\end{proof}

 % On the boundary $|\nabla \phi |$ can be estimated by $ \nabla \phi_0 $ and $ \nabla \phi_1 $.  

\section{$\p_{tt}$ estimate}
\label{S}

 Let $Q = \p_{ t t } +  ( \ul{ \p } - \p )$ attain maximum at $z_0$ in the interior of $Y$. Then

\begin{equation}\label{S1}
\begin{aligned}
    F^{\al \obe} \p_{ t t \al \obe} +  F^{\al \obe } ( \ul{\p} - \p )_{ \al \obe } \leq 0
\end{aligned}   
\end{equation}

\noindent and 

\begin{equation}\label{compatibility}
    \p_{ttt} = - (\ul{\p}_t - \p_t),
\end{equation}

\noindent which implies $\p_{ttt}$ is uniformly bounded at the point $z_0$, from Section \ref{G}.

Write the equation \eqref{C3.1} as

\begin{equation}
    \log(\phi_{ t t } A(\p ) - \sum_k |\p_{ k t}|^2) = \log{\left(\epsilon - \frac{n X \p_t^2}{2}\right)}
\end{equation}
\noindent and differentiate by $\partial_t \partial_t$ to get

\begin{equation}\label{S2}
    \begin{aligned}
        \frac{1}{G(\p)}& F^{\al \obe} \p_{ t t \al \obe} + \frac{1}{G(\p)} F^{\al \obe, \ga \ode} \p_{\al \obe t} \p_{\ga \ode t}  + \frac{1}{G(\p)} (n X \p_{t t}^2  + 2n X \phi_t \phi_{ t t t })\\
        &-\frac{1}{G(\p)^2}\left(n X \p_{tt} \p_t + F^{\alpha \bar{\beta}} \p_{\alpha \bar{\beta} t }\right)^2= \frac{ 1 }{ L } L_{ t t } - \frac{1}{L^2} L_t^2 
    \end{aligned}
\end{equation}

We compute

\begin{equation} \label{S3}
    L_{ t t } = - n X \p_{ t t }^2 - n X \phi_t \phi_{ t t t }, \hspace{1cm}
    L_{ t }^2 = ( n X \p_{ t t } \phi_t )^2.
\end{equation}

Now there is an important cancellation between terms that are quadratic in $\phi_{ t t } $.

\begin{equation} \label{S5}
\begin{aligned}
   \frac{ 1 }{ L } L_{ t t } - \frac{1}{L^2} L_t^2 - \frac{1}{G(\p)} (n X \p_{t t}^2  + 2n X \phi_t \phi_{ t t t })=
   \frac{ -2\epsilon n X \phi_{ t t }^2 -3 \epsilon n X \p_t \p_{ttt} + (3/2) n^2 X^2 \p_t^3 \p_{ttt}  }{L^2} \\
   \end{aligned}
\end{equation}

Here we used the equation $G(\p) = L$. By concavity of $F$ and plurisubharmonicity of \eqref{plurisub}, the second term in \eqref{S2} is negative. Hence we get that 

\begin{equation}\label{S6}
F^{\al \obe } \p_{ t t \al \obe} \geq \frac{ -3 \epsilon n X \p_t \p_{ttt} + (3/2) n^2 X^2 \p_t^3 \p_{ttt}  }{L} \geq -6 n \sup\limits_{Y}|X\p_t\p_{ttt}|
\end{equation}

\noindent where we used $L \geq \min\{\epsilon, -nX\p_t^2/2\}$. This is a bounded quantity. Since $A(\p) > 0$, by assuming that $\phi_{tt} \gg 1$ at $z_0$, from \eqref{C5.4} it is clear that

\begin{equation}\label{S7}
 F^{\al \obe } ( \ul{\p} - \p )_{ \al \obe } \gg 1   
\end{equation}

Inequalities \eqref{S6} and \eqref{S7} will together contradict \eqref{S1}. Hence the maximum for $Q$ must be attained at the boundary of $Y$.

\begin{equation} \label{S8}
   \sup_{Y} \p_{t t } \leq C \sup_{\partial Y}  ( 1 + \phi_{ t t } ) 
\end{equation}

\section{ $C^2$ Boundary estimates}
\label{B}

Denote $K = \sup\limits_{\partial Y}{(1 + |\nabla u|^2)}$. Then $K$ is bounded because of the boundary conditions. It is enough to show that $| \p_{ k t } | \leq C K$ on the boundary. Estimates for $\p_{ t t }$ and $ \Delta \p $ on the boundary will follow from the equation and the boundary conditions respectively.

We derive the estimate around a boundary point corresponding to $ t = 0 $ by constructing a local barrier function. The $ t = 1 $ case can be done similarly. Let $ z_0 $ be any point on the boundary at $ t = 0 $. Consider a coordinate ball $B_{\de}$ centered at $ z_0 = 0 $ of radius $ \delta $. Then define a barrier function $h$ on $ \Omega_{\de} = B_{\de} \times [ 0 , t_0]$ with $0 < t_0 < 1 $ given by

$$h = A_1( \p - \ul{\p} ) + B |z|^2 + C( t - t^2 ) + ( \p - \ul{\p} )_k $$

\

\noindent where $A_1$, $B$, $C$ are large multiples of $K$ to be fixed later. Let $C_1$, $C'$ and $c_0$ denote independent uniform constants.

First we show that $h \geq 0$ on $\partial \Omega_{\de}$. There are three cases.

\begin{enumerate}
    \item  If $t=0$, then $h= B|z|^2 \geq 0 $.

    \item If $t = t_0$, then 

    $$h = A_1(\p - \ul{\p} ) + B |z|^2 + C(t_0 - t_0^2) + (\p - \ul{\p})_k \geq 0$$

    \noindent for $ C \gg K $.

    \item If $z \in \partial B_{\de}$, then 

    $$h = A_1(\p - \ul{ \p } ) + B \de^2 + C( t - t^2) + (\p - \ul{\p})_k \geq 0$$

    \noindent for $B \gg K $.
    
\end{enumerate}

Now we compute $ \mathcal{L} h $.  From \eqref{C5.4}, it follows that 

\begin{equation} \label{S19}
    \mathcal{L} (\p - \ul{ \p } ) < - \epsilon \sum F^{\al \oal} = - \epsilon C_1 ( \p_{ t t } + A( \p ) )
\end{equation}

\begin{equation} \label{S20}
    \mathcal{L} ( B|z|^2 + C (t - t^2 ) ) = 2 n B \p_{ t t }  -2 C A(\p) 
\end{equation}

By differentiating equation \eqref{pde} we also get

\begin{equation} \label{S21}
    | \mathcal{L} ( \p -\ul{ \p} )_k | \leq C_1( |\p_{ t t }| + |\p_{ t k }| )
\end{equation}

From \cref{S19,S20,S21}, it follows that $\mathcal{L}(h) \leq 0 $ for $A_1$ much larger than both $B$, $C$ and other bounded constants. Hence by maximum principle, $h \geq 0$ on $\Omega_{\delta}$. Since $h(0) = 0$, and

$$\frac{\partial h}{\partial t } \geq 0 $$

\noindent at $z_0$, it follows that $- \p_{ t k }(0) \leq C' K$. Similarly by considering $A_1(\p - \ul{ \p } ) + B |z|^2 + C( t - t^2) - (\p - \ul{\p})_k$, we also get $\p_{ t k }(0) \leq C' K$. As a result $|\p_{ t k }(0)| \leq C'K$.

To bound $\p_{ t t }(0)$, notice that $A(\p) = A(\p_0)> c_0 >0$ at $0$ and hence $\p_{ t t } = \dfrac{ \sum |\p_{k t}|^2 + L}{A( \p )} $ is bounded at $0$. This gives the boundary estimates and can write

$$\sup_{\partial Y}~ (|\phi_{ t t }| + |\Delta \p | ) \leq C' $$

\

\section{Calabi-Yau theorem for balanced metrics}
\label{CYB}

Since S.-T. Yau \cite{Yau78} proved the Calabi conjecture in 1976, there has been great interest in establishing similar theorems in non-K\"ahler geometry. That is, to show the existence of special Hermitian metrics with prescribed Chern-Ricci forms. 

An important result along these lines is the Gauduchon conjecture \cite{Gauduchon84} that was resolved in \cite{STW17}. Perhaps a more interesting theorem from the perspective of geometry and mathematical physics would be to solve the same problem for balanced metrics. We refer to the related work of Fu-Wang-Wu \cite{FWW10} to how this implies the conformal balanced equation in the Strominger system. This asks for a Hermitian metric $\omega'$ such that 

\begin{equation}\label{conformal}
d(||\Omega||_{\omega'} {\omega'}^2) = 0
\end{equation}

\noindent for a non-vanishing holomorphic $(3,0)$ form $\Omega$ on a $3$ dimensional Hermitian manifold $(M, \omega)$. Fixing $||\Omega||_{\omega'} = C$, \eqref{conformal} says $\omega'$ is a balanced metric. Then

$$\frac{{\omega'}^3}{\omega^3} = \frac{||\Omega||^2_{\omega}}{||\Omega||^2_{\omega'}} = e^{k}$$

\noindent for some function $k$. This has been studied by Fu-Wang-Wu, solving it on a flat torus in dimension $3$ by considering some explicit parametrizations of the metric \cite{FWW10}, and also on K\"ahler manifolds under non-negative curvature assumptions \cite{FWW15}. The equation we consider is a bit different and will use more general cohomology relations similar to \cite{STW17}.

Given a balanced metric $\omega$ and a background Hermitian metric $\alpha$, define a new metric $\omega_u$ by

\begin{equation}
    \omega_u^{n-1} = \omega^{n-1} + \sqrt{-1} \pbp (u\alpha^{n-1})
\end{equation}

Clearly, $\omega_u$ is also balanced as $d \omega_u^{n-1} = 0$. Given a $(1,1)$ form $\Psi$ in $c_1^{BC}(M)$, we look for an unknown function $u$, such that

\begin{equation}
    \mathrm{Ric}^C(\omega_u) = -\sqrt{-1} \pbp \log \det \omega_u = \Psi
\end{equation}

Written as a PDE in local coordinates this becomes

\begin{equation}\label{CYb}
    \det \left(\omega_h + \frac{1}{(n-1)}\left(\Delta_{\alpha} u ~\alpha - \sqrt{-1} \pbp u\right) + \chi(\partial u, \bpartial u) + E u \right) = e^{\psi} \det \alpha
\end{equation}

\noindent where $\omega_h = \star \omega^{n-1}$, $\chi(\partial u, \bpartial u)$ is smooth $(1,1)$ form involving the torsion tensor and is linear in $\nabla u$, and 

$$E= \star \sqrt{-1} \pbp \alpha^{n-2}$$

\noindent is a $(1,1)$ form. Here $\star$ is the Hodge star operator with respect to the metric $\alpha$. We refer to \cite{STW17} for this transformation and the exact form of $\chi(\partial u, \bpartial u)$. First we make the following defintions.

\begin{Def} A Hermitian metric $\alpha$ is called 

\

\begin{enumerate}[(i)]
    \item { Sub-Astheno-K\"ahler} if $\star \sqrt{-1} \pbp \alpha^{n-2} \leq 0$.

    \
    
    \item { Super-Astheno-K\"ahler} if $\star \sqrt{-1} \pbp \alpha^{n-2} \geq 0$.
\end{enumerate}
\end{Def}

A Hermitian manifold is sub- or super-Astheno-K\"ahler if it admits such a metric.

Define the metric

\begin{equation}\label{u-metric}
\tilde{\omega}_{u} =\omega_h + \frac{1}{(n-1)}\left(\Delta u \alpha - \sqrt{-1} \pbp u\right) + \chi(\partial u, \bpartial u) + E u > 0 .
\end{equation}

We show that assuming that $E \leq 0$ as a $(1,1)$ form is sufficient for obtaining $C^0$ estimates for this equation. This will done in two different ways. First we use the ABP maximum principle introduced by Blocki \cite{Blocki11} for complex Monge-Amp\`ere equation and then extended to general fully nonlinear case admitting a $\mathcal C$-subsolution by Sz\'ekelyhidi \cite{Szekelyhidi18}. The second proof will be based on the technique using the auxilliary Monge-Amp\`ere equation. Note that the $\alpha$-trace of $E$ is exactly the quantity $X$ (for $\alpha$) when $p=n-1$ from the previous sections.

$$X = \mathrm{tr}_{\alpha} E$$

\begin{theorem}\label{theorem-CYb}
    Let $(M, \omega)$ be a compact balanced manifold. Assume that $E \leq 0$ on $M$ for some Hermitian metric $\alpha$. Then for any solution $u$ of equation \eqref{CYb}
    
    $$\sup_M ~ |u| \leq C,$$
    \noindent for a uniform constant $C$ that depends only on $(M,\omega)$, $\psi$, and $\alpha$.
\end{theorem}

We remark that due to the works of Tosatti-Weinkove \cite{TW17, TW19}, Sz\'ekelyhidi-Tosatti-Weinkove \cite{STW17}, and Guan-Nie \cite{GN23}, it is enough to derive $C^0$ estimates, as all the other estimates will follow similar to those calculations. The $C^2$ estimates can be obtained by simple modifications in the proof of Gauduchon conjecture \cite{STW17}, or the work of Guan-Nie \cite{GN23}. Hence Theorem \ref{theorem-CYb} gives an existence result for equation \eqref{CYb} satisfying the condition $E\leq 0$. In fact, the solution will be unique under this assumption. From \cite{Szekelyhidi18}, we recall the notion of a $\cC$-subsolution, originally introduced by Guan \cite{Guan14}.

\begin{Def}
    Suppose that $(M, \alpha)$ is a Hermitian manifold and $\omega_h$ is a real $(1,1)$ form. We say that a smooth function $\ul{u}$ is a $\cC$-subsolution for 
    $$f(\lambda(\alpha^{i \bk}(\tilde{\omega}_{{u}})_{j \bk})) = h$$
    if at each $x \in M$, the set

    $$\{\lambda[\alpha^{i \bk} (\tilde{\omega}_{\ul{u}})_{j \bk}] + \Gamma_n \} \cap \partial \Gamma^{h(x)}$$

    \noindent is bounded.
\end{Def}

Here $f$ is a symmetric function of eigenvalues with standard structure assumptions of \cite{CNS85}. From the definition it can be seen that $0$ is a $\mathcal{C}$-subsolution to $\eqref{CYb}$. A smooth function $\ul{u}$ being a $\mathcal C$-subsolution implies that there exists a $\delta > 0$ and $R >0$ such that at each $x$

\begin{equation}\label{Csub}
    \{\lambda[\alpha^{i \bk} (\tilde{\omega}_{\ul{u}})_{j \bk}] - \delta \mathbf{1} + \Gamma_n \} \cap \partial \Gamma^{h(x)} \subset B_R(0)
\end{equation}

\

\begin{proof}[Proof of Theorem \ref{theorem-CYb}]

The proof uses a version of the ABP maximum principle. 
% Note that $\underline{u} = 0$ is $\mathcal C$-subsolution for the equation \eqref{CYb}, as the RHS is uniformly bounded.

First observe that $\mathrm{tr}_{\alpha}\tilde{\omega}_u > 0$, would give the following elliptic equation. 

\begin{equation}\label{trace}
    \Delta_{\alpha} u + \mathrm{tr}_{\alpha} \chi(\nabla u) + X u + f(z) > 0
\end{equation}

\noindent for some function $f$. Then from linear elliptic theory (Theorem $3.7$ in \cite{GT83}), using that $X \leq 0$, we obtain estimates for the supremum of the function $u$.

$$\sup_M u \leq C.$$

\noindent for $C$ depending on $f$ and the coefficients of the equation. So it is enough to estimate $\inf\limits_M u$. Let $ m = \inf\limits_M u$ be attained at a point $z_0$ on $M$ and assume that $m < 0$. Choose local coordinates that takes $z_0$ to the origin and consider the coordinate ball $B(1) = \{z : |z| <1\}$, chosen small enough so that $u\leq 0$ on $B(1)$. Let $v = u + \kappa |z|^2$ for a small $\kappa > 0$, so that $\inf\limits_{B(1)} v = m = v(0)$, and $\inf\limits_{z \in \partial B(1)} v(z) \geq m + \kappa$.

Then by the ABP maximum principle for upper contact sets (Chapter $9$ of \cite{GT83}, Proposition $10$ in \cite{Szekelyhidi18}), the set 

$$\Gamma^+_{\kappa} = \{x \in B(1): |Dv(x)| \leq \frac{\kappa}{2} \text{ and } v(y) - v(x) \geq Dv(x).(y-x) \text{ for all } y \in B(1)  \}$$

\noindent satisfies 

\begin{equation}\label{ABPint}
    \int_{\Gamma^+_{\kappa}} \det{D^2 v} \geq c_0 \kappa^{2n}. 
\end{equation}
\noindent for some positive constant $c_0$. The following can be verified at any $x \in \Gamma^+_{\kappa}$ as in \cite{Blocki11},

\begin{enumerate}
    \item $D^2v(x) \geq 0$.
    \item $\det(D^2v) \leq 2^{2n} (\det v_{i \bj})^2$.
    \item $u_{i \bj} (x) \geq - \kappa \delta_{i j}$.
\end{enumerate}

From this and using $E\leq 0$ we can conclude that in the set $\Gamma_{\kappa}^+$

\begin{equation}
    \tilde{\omega}_u \geq \omega_h + \frac{1}{(n-1)}\left(\Delta u \alpha - \sqrt{-1} \pbp u\right) - \epsilon' \alpha
\end{equation}

\noindent for a small $\epsilon'>0$ that depends on $\kappa$, which is fixed small enough depending on $\alpha$. 

% Denoting \red{do i need this?}
% $$\beta = \omega_h + \frac{1}{(n-1)}\left(\Delta u \alpha - \sqrt{-1} \pbp u\right) - \epsilon' \alpha,$$

It follows that at $x \in \Gamma_{\kappa}^+$

$$\lambda [\alpha^{i \bk} (\tilde{\omega}_u)_{j \bk}] \in \{ \lambda[\alpha^{i \bk } (\omega_h)_{i \bk}] - \delta \mathbf{1} + \Gamma_n \}$$

\noindent for $\delta$ small depending on $\epsilon'$, and $\kappa$.

From the equation \eqref{CYb}, it is clear that $\lambda [\alpha^{i \bk} (\tilde{\omega}_u)_{j \bk}] \in \partial \Gamma^\sigma$ for $\sigma = e^{\psi} \det \alpha$. Hence we get from \eqref{Csub} that 

$$|u_{i \bj}| \leq C.$$

This gives a uniform bound for $v_{i \bj}$ in $\Gamma_{\kappa}^+$. Then it follows from the above using \eqref{ABPint} 

$$c_0 \kappa^{2n} \leq C' \mathrm{vol}(\Gamma_{\kappa}^+),$$

\noindent for some constant $C' >0$. From the weak Harnack inequality (Theorem 8.18 in \cite{GT83}) applied to equation \eqref{trace}, on $B(1)$

$$  \int_{B(1)} |u|^p \leq C (1 + \inf_{B(1)}|u|) \leq C'$$

This implies that $|v|_{L^p}$ is bounded. In the set $\Gamma_{\kappa}^+$, $v(x) \leq v(0) + \dfrac{\kappa}{2}$. Putting these together with 
\begin{equation}
    \mathrm{vol}(\Gamma^+_{\kappa}) \left|v(0) + \frac{\kappa}{2} \right|^p ~\leq ~\int_{\Gamma^+_{\kappa}} |v|^p ~\leq ~C,
\end{equation}

\noindent we get $|m + \kappa/2|^p \leq C \kappa^{-2n}$, which shows that $m$ is bounded.

\end{proof}

\begin{Rmk}
    The technique from above extends directly to the case of fully nonlinear equations of symmetric functions of eigenvalues under the $\mathcal C$-subsolution condition. Hence one can consider equations of the form

    \begin{equation}
        f(\lambda(\omega + \sqrt{-1} \pbp u + \chi(\partial u, \bpartial u) + E u)) = \psi(z) 
    \end{equation}
    \noindent for some $(1,1)$ form $E \leq 0$.
\end{Rmk}

\begin{Rmk}
  Examples of non-K\"ahler balanced manifolds that admit a Hermitian metric $\alpha$ so that $E \leq 0$ can be inferred from \cite{LU17}, using explicit constructions on complex nilmanifolds.  
\end{Rmk}

Next we show how the $C^0$ estimates can be obtained under a weaker assumption that $\mathrm{tr}_{\tilde{\omega}_u}E \leq 0$, by using the auxilliary Monge-Amp\`ere equation. The method of auxilliary Monge-Ampere equation, inspired by the works of Chen-Cheng \cite{CC21} on cscK metrics was developed by Guo-Phong-Tong \cite{GPT23} to give PDE proofs for $L^{\infty}$ estimates to the Monge-Amp\`ere equation, when the right-hand side is only $L^p$ for any $p>1$. This method was later extended to the case of more general fully nonlinear equations in \cite{GP23} and equations involving $(n-1)$ forms and gradient terms in \cite{GP23-1}. We show that this can be used to obtain $C^{0}$ estimates for equation \eqref{CYb} under a weaker assumption. The method is similar and we only show the part of the proof that obtains a comparison between solution of equation \eqref{CYb} and the solution of the auxilliary Monge-Amp\`ere equation. The argument for equation seen in Gauduchon conjecture has been treated in \cite{GP23-1}, where to deal with the gradient terms in this equation, Guo and Phong considers the real Monge-Amp\`ere equation for comparison, for which independent gradient estimates are available. Also see the related work of Klemyatin-Liang-Wang \cite{KLW23}.

We write \eqref{CYb} as
\begin{equation}
    \log{\det(\tilde {\omega}_u)} = \psi + \log{ \det{\alpha}}.
\end{equation}

The principal part of the linearization of this operator is given by

$$ Dv = \frac{1}{n-1} \left( \mathrm{tr}_{G} \alpha {g'}^{i \bj} - G^{i \bj} \right) \partial_i \bpartial_j v = D^{i \bj }  \partial_i \bpartial_j v,$$

\noindent where for simplicity we used $G^{i \bj} = \tilde {\omega}_u^{i \bj}$ and $g'$ is the metric corresponding to $\alpha$.

It is not too difficult to show that $D^{i \bj}$ is positive definite at a solution $u$, and 

$$\det(D^{i \bj}) > \gamma >0$$

\noindent for a constant $\gamma$ that depends only on the right-hand side of equation \eqref{CYb}.

Assume that $u$ attains a negative minimum at $z_0$. Let $B_{2r}$ denote the open ball of radius $2r$ around $z_0$ small enough so that the metric $\alpha$ is close to the euclidean metric in this ball, and such that $u\leq 0$ in $B_{2r}$. Similarly $B_r$ is a the ball of radius $r$ centered at $z_0$. The auxilliary real Monge-Amp\'ere equation is given by

\begin{equation}\label{aux}
    \det{\left(\frac{\partial^2 w_{s,k}}{\partial x_a \partial x_b}\right)} = \frac{\tau_k(-u_s)}{A_{s,k}} e^{\sigma} (\det g'_{i \bj})^2 
\end{equation}

\noindent where $\tau_k(x)$ is a sequence of smooth functions on $\bbR$ that converges to the $x .\chi_{\bbR^+}$ from above, $A_{s,k}$ is normalization constant so that the integral of the RHS is $1$, and $\sigma = 2n(\psi + \log(\det{\alpha}))$.

The solution $w_{s,k}$ with the boundary condition that $w_{s,k} = 0$ on $\partial B_{2r}$ exists, is bounded, and has a uniform bound for the gradient $|\nabla w_{s,k}|$ (Lemma $11$ in \cite{GP23-1}). Let $\epsilon'>0$ be a small constant, and consider for any $s \in (0, \epsilon' r^2)$, the function $u_s(z) = u(z) - u(z_0) + \epsilon'|z|^2 - s $, and the function

$$\varphi = -\tilde{\epsilon}(-w_{s,k} + \Lambda)^{\frac{2n}{2n+1}} - u_s.$$

Here $\Lambda$ and $\tilde{\epsilon}$ are constants to be chosen later that depends on the bounds on $|w_{s,k}|$, $|\nabla w_{s,k}|$ and other known quantities. Note that $w_{s,k}$ is a convex function on $B_{2r}$ and hence is non-positive. It can be immediately observed that $u_s \geq 0$ on on $B_{2r}\setminus B_r$ and hence the set $\Omega_s = \{z \in B_{2r} |~ u_s(z) <0\}$ is contained in $B_r$.

The following formula follows directly by taking the $\tilde{\omega}_u$-trace of the definition \eqref{u-metric} of the metric $\tilde{\omega}_u$.
\begin{equation}\label{cone-ineq}
    -D^{i \bj} u_{i \bj} = -n + \mathrm{tr}_{G} \omega_h + \mathrm{tr}_G \chi + (\mathrm{tr}_{G} E) u
\end{equation}

\noindent for a small constant $c_0>0$. To show that $\varphi \leq 0$ in $B_{2r}$, the point $z_1$ where $\varphi$ attains a maximum can be assumed to be in $ B_{r}$. At $z_1$, the compatibility equation $\nabla\varphi = 0$ implies
\begin{equation}\label{torsion}
\begin{aligned}
   |G^{i \bj}\chi_{i \bj}(\partial u, \bpartial u)|&\leq C\left|G^{i \bj}T_{\bj}\left(\frac{2n\tilde{\epsilon}}{2n+1}(-w_{s,k}+ \Lambda)^{-\frac{1}{2n+1}}(w_{s,k})_i - \epsilon' \bar{z}_i\right) \right|\\ 
   & \leq c_0 \mathrm{tr}_G\omega_h \\
   \end{aligned}
\end{equation}

\noindent where we use $T_{\bj}$ to denote the torsion coefficients that appear in $\chi(\partial u, \bpartial u)$, and $c_0$ is a small positive constant. In the last line, $\Lambda$ and $\epsilon'$ are chosen depending on $|{\nabla w_{s,k}}_i|$, $\omega_h$, $\tilde{\epsilon}$, and other background data to get this bound. Taking second derivatives of $\varphi$, we have at $z_1$ 

\begin{equation}
    \begin{aligned}
        0 &\geq D^{i \bj} \varphi_{i \bj} \\ &\geq  \frac{2n \tilde{\epsilon}}{2n+1} (-w_{s,k} + \Lambda)^{-\frac{1}{2n+1}} D^{i \bj} (w_{s,k})_{i \bj} - n + \mathrm{tr}_{G} \omega_h + \mathrm{tr}_G \chi - \epsilon' \mathrm{tr}_G \alpha + (\mathrm{tr}_G E)  u \\
        & \geq  \frac{2n^2 \tilde{\epsilon} \gamma^{1/n}}{2n+1} (-w_{s,k} + \Lambda)^{-\frac{1}{2n+1}} \frac{(-u_s)^{1/2n}e^{\frac{\sigma}{2n}}\det(g'_{i \bj})^{\frac{1}{n}}}{A_{s,k}^{1/2n}} - n + (\mathrm{tr}_G E) u
    \end{aligned}
\end{equation}

\noindent where we used equation \eqref{cone-ineq}, 

$$D^{i \bj }(w_{s,k})_{i \bj} \geq (\det{D^{i \bj}})^{1/n} (\det(w_{s,k})_{i \bj})^{1/n}, $$

\noindent and the auxilliary equation \eqref{aux}. We also used $(2)$ from proof of Theorem \ref{theorem-CYb}, and \eqref{torsion} here. Since by assumption $(\mathrm{tr}_G E)u \geq 0$, by a clever choice of $\tilde{\epsilon}$ in this equation we get that $\varphi(z_1) \leq 0 $.  To find the exact bound on $\tilde{\epsilon}$, set

$$r = \inf\limits_{M} \left(\frac{2n^2 \gamma^{1/n}e^{\sigma/2n} \det{(g'_{i \bj})^{1/n}}}{(2n+1) A_{s,k}^{1/2n}}  \right) > 0.$$

Then take $\tilde{\epsilon} \geq \left(\dfrac{n}{r}\right)^{\dfrac{2n}{2n+1}}$.

\

This gives a comparison between $u(z_0)$ and $w_{s,k}$. From here the argument is identical to \cite{GPT23}, and we can obtain $-\inf\limits_M u \leq C$, for a constant that depends on the background data, and the entropy of the function $e^{\psi} \det(\alpha)$.

Finally, we add some remarks on the openness part of the continuity method and uniqueness of the solution to equation \eqref{CYb}. The openness can be shown by applying inverse function theorem to the linearized operator. This follows the same steps from \cite{TW17} and we skip it here. For uniqueness, we could assume that there exist two distinct pairs $(b_1, u_1)$ and $(b_2, u_2)$ that solves the equation. Taking the difference of the equations and applying maximum principle as in \cite{TW17} proves the uniqueness up to a constant. The only additional comment is that in \cite{TW17}, the solutions were normalized to $\sup\limits_M{u} = 0$, otherwise it is only unique up to a constant. This is not possible in our case because of the term $Eu$. But at the same time, if $E \not\equiv 0$, then this condition is unnecessary as a translation of the solution will no longer solve the equation.

\section{Final Remarks}
\label{FR}
We state a few questions here that would be interesting to study further.

\

\begin{enumerate}[(i)]

\item It is unclear how to obtain interior $C^2$ estimates for equation \eqref{pde}. That is, estimates for $|\sqrt{-1}\pbp \p|$. The degeneracy of the equation, in addition to the function $X$ poses some difficulties.

\

    \item It would be interesting to investigate manifolds that admit Hermitian metrics such that $\pbp \alpha^{n-2} \leq 0 $. Are there any strong topological restrictions for this condition? For example 

$$\pbp \alpha^{n-2} = 0$$ 

\noindent is the astheno-K\"ahler condition of Jost-Yau \cite{JY93}, that is known to have some obstructions \cite{CR23}. It was shown by Jost and Yau that holomorphic $1$-forms on an astheno-K\"ahler manifold are $\partial$-closed. In fact, this holds under the weaker assumption $\pbp \alpha^{n-2} \geq 0$ by similar argument. If $\gamma$ is a holomorphic $1$-form, then

$$0 \geq \int \partial \gamma \wedge \bpartial \ol{\gamma} \wedge \alpha^{n-2} = \int \gamma \wedge \ol{\gamma} \wedge \pbp \alpha^{n-2} \geq 0 $$

This would imply that $\partial \gamma = 0$.

\

\item The condition $\mathrm{tr}_{\tilde{\omega}_u} E \leq 0$ suggests considering the continuity path

$$\Gamma_{\kappa} = \{u\in C^4(M): \tilde{\omega}_u >0 \text{ and } \mathrm{tr}_{\tilde{\omega}_u} E \leq \kappa <0 \}$$

This would need one to derive independent upper bound for $\mathrm{tr}_{\tilde{\omega}_u} E$ for a solution $u$. Such an estimate could solve the equation \eqref{CYb} under the weaker assumption that there exist a balanced metric $\omega$ and a Hermitian metric $\alpha$ such that

$$ \mathrm{tr}_{\star\omega^{n-1}}( \star\sqrt{-1} \pbp \alpha^{n-2} )\leq 0,$$

\noindent where $\star$ is with respect to $\alpha$. If $\alpha = \omega$, this reduces to $X$ from the geodesic equation above which clearly cannot be negative. But for a general $\alpha$ this might be admissible in all balanced manifolds.

\

\item Due to the similarities to the K\"ahler-Einstein equation, it makes sense to conjecture that \eqref{CYb} might not admit solutions, in general, if $E \geq 0$. 
    
\end{enumerate}

\

% (ii) It would also make sense to look for counterexamples to the solvability of equation \eqref{CYb} when $E \geq 0$. From the intuition given by linear theory and the theory of K\"ahler-Einstein metrics this equation might not admit solutions always. Are there some geometric obstructions that can be obtained in this case?

\

\bibliographystyle{plain}
\bibliography{references}

@incollection {AB96,
    AUTHOR = {Alessandrini, Lucia and Bassanelli, Giovanni},
     TITLE = {The class of compact balanced manifolds is invariant under
              modifications},
 BOOKTITLE = {Complex analysis and geometry ({T}rento, 1993)},
    SERIES = {Lecture Notes in Pure and Appl. Math.},
    VOLUME = {173},
     PAGES = {1--17},
 PUBLISHER = {Dekker, New York},
      YEAR = {1996},
      ISBN = {0-8247-9672-1},
   MRCLASS = {32J27 (32C10)},
  MRNUMBER = {1365967},
MRREVIEWER = {Shanyu\ Ji},
}

@article {Blocki11,
    AUTHOR = {B\l ocki, Zbigniew},
     TITLE = {On the uniform estimate in the {C}alabi-{Y}au theorem, {II}},
   JOURNAL = {Sci. China Math.},
  FJOURNAL = {Science China. Mathematics},
    VOLUME = {54},
      YEAR = {2011},
    NUMBER = {7},
     PAGES = {1375--1377},
      ISSN = {1674-7283,1869-1862},
   MRCLASS = {32Q25 (32W20 53C55)},
  MRNUMBER = {2817572},
MRREVIEWER = {Bianca\ Santoro},
       DOI = {10.1007/s11425-011-4197-6},
       URL = {https://doi.org/10.1007/s11425-011-4197-6},
}

@article {CNS85,
    AUTHOR = {Caffarelli, L. and Nirenberg, L. and Spruck, J.},
     TITLE = {The {D}irichlet problem for nonlinear second-order elliptic
              equations. {III}. {F}unctions of the eigenvalues of the
              {H}essian},
   JOURNAL = {Acta Math.},
  FJOURNAL = {Acta Mathematica},
    VOLUME = {155},
      YEAR = {1985},
    NUMBER = {3-4},
     PAGES = {261--301},
      ISSN = {0001-5962,1871-2509},
   MRCLASS = {35J65 (53C40 58G30)},
  MRNUMBER = {806416},
MRREVIEWER = {Philippe\ Delano\"{e}},
       DOI = {10.1007/BF02392544},
       URL = {https://doi.org/10.1007/BF02392544},
}

@article {Chen2000,
    AUTHOR = {Chen, Xiuxiong},
     TITLE = {The space of {K}\"{a}hler metrics},
   JOURNAL = {J. Differential Geom.},
  FJOURNAL = {Journal of Differential Geometry},
    VOLUME = {56},
      YEAR = {2000},
    NUMBER = {2},
     PAGES = {189--234},
      ISSN = {0022-040X,1945-743X},
   MRCLASS = {32Q15 (32W20 58D27 58E11)},
  MRNUMBER = {1863016},
MRREVIEWER = {David\ L.\ Finn},
       URL = {http://projecteuclid.org/euclid.jdg/1090347643},
}

@article {CH11,
    AUTHOR = {Chen, Xiuxiong and He, Weiyong},
     TITLE = {The space of volume forms},
   JOURNAL = {Int. Math. Res. Not. IMRN},
  FJOURNAL = {International Mathematics Research Notices. IMRN},
      YEAR = {2011},
    NUMBER = {5},
     PAGES = {967--1009},
      ISSN = {1073-7928,1687-0247},
   MRCLASS = {58E10 (53C23 58B20 58D17)},
  MRNUMBER = {2775873},
MRREVIEWER = {Gabjin\ Yun},
       DOI = {10.1093/imrn/rnq099},
       URL = {https://doi.org/10.1093/imrn/rnq099},
}

@article {CH19,
    AUTHOR = {Chen, Xiuxiong and He, Weiyong},
     TITLE = {A class of fully nonlinear equations},
   JOURNAL = {Pure Appl. Math. Q.},
  FJOURNAL = {Pure and Applied Mathematics Quarterly},
    VOLUME = {15},
      YEAR = {2019},
    NUMBER = {4},
     PAGES = {1029--1045},
      ISSN = {1558-8599,1558-8602},
   MRCLASS = {35J60 (35B45)},
  MRNUMBER = {4085666},
MRREVIEWER = {Gabrielle\ Nornberg},
       DOI = {10.4310/PAMQ.2019.v15.n4.a3},
       URL = {https://doi.org/10.4310/PAMQ.2019.v15.n4.a3},
}

@article {Chu19,
    AUTHOR = {Chu, Jianchun},
     TITLE = {{$C^{1,1}$} regularity of geodesics in the space of volume
              forms},
   JOURNAL = {Calc. Var. Partial Differential Equations},
  FJOURNAL = {Calculus of Variations and Partial Differential Equations},
    VOLUME = {58},
      YEAR = {2019},
    NUMBER = {6},
     PAGES = {Paper No. 194, 7},
      ISSN = {0944-2669,1432-0835},
   MRCLASS = {58E10 (35B65 35J60 35J70 35R01 58D17)},
  MRNUMBER = {4024609},
MRREVIEWER = {Rossella\ Bartolo},
       DOI = {10.1007/s00526-019-1648-3},
       URL = {https://doi.org/10.1007/s00526-019-1648-3},
}

@unpublished{Demailly12,
  author    = {Jean-Pierre Demailly},
  title     = {Complex Analytic and
Differential Geometry},
  year      = {2012},
  url       = {https://www-fourier.ujf-grenoble.fr/~demailly/manuscripts/agbook.pdf}
}

@incollection {Donaldson10,
    AUTHOR = {Donaldson, Simon K.},
     TITLE = {Nahm's equations and free-boundary problems},
 BOOKTITLE = {The many facets of geometry},
     PAGES = {71--91},
 PUBLISHER = {Oxford Univ. Press, Oxford},
      YEAR = {2010},
      ISBN = {978-0-19-953492-0},
   MRCLASS = {58E15 (35R35 70S15)},
  MRNUMBER = {2681687},
MRREVIEWER = {Derek\ G.\ Harland},
       DOI = {10.1093/acprof:oso/9780199534920.003.0005},
       URL = {https://doi.org/10.1093/acprof:oso/9780199534920.003.0005},
}

@incollection {Donaldson99,
    AUTHOR = {Donaldson, S. K.},
     TITLE = {Moment maps and diffeomorphisms},
      NOTE = {Sir Michael Atiyah: a great mathematician of the twentieth
              century},
   JOURNAL = {Asian J. Math.},
  FJOURNAL = {Asian Journal of Mathematics},
    VOLUME = {3},
      YEAR = {1999},
    NUMBER = {1},
     PAGES = {1--15},
      ISSN = {1093-6106,1945-0036},
   MRCLASS = {53D20 (32Q25 37K65 53C26 53C38 53D12 53D35)},
  MRNUMBER = {1701920},
MRREVIEWER = {Matthew\ B.\ Stenzel},
       DOI = {10.4310/AJM.1999.v3.n1.a1},
       URL = {https://doi.org/10.4310/AJM.1999.v3.n1.a1},
}

@incollection {Fu16,
    AUTHOR = {Fu, Jixiang},
     TITLE = {A survey on balanced metrics},
 BOOKTITLE = {Geometry and topology of manifolds},
    SERIES = {Springer Proc. Math. Stat.},
    VOLUME = {154},
     PAGES = {127--138},
 PUBLISHER = {Springer, [Tokyo]},
      YEAR = {2016},
      ISBN = {978-4-431-56021-0; 978-4-431-56019-7},
   MRCLASS = {32J27 (53C15 53C55)},
  MRNUMBER = {3555979},
MRREVIEWER = {Ahmed\ Lesfari},
       DOI = {10.1007/978-4-431-56021-0\_6},
       URL = {https://doi.org/10.1007/978-4-431-56021-0_6},
}

@article {FWW15,
    AUTHOR = {Fu, Jixiang and Wang, Zhizhang and Wu, Damin},
     TITLE = {Form-type equations on {K}\"{a}hler manifolds of nonnegative
              orthogonal bisectional curvature},
   JOURNAL = {Calc. Var. Partial Differential Equations},
  FJOURNAL = {Calculus of Variations and Partial Differential Equations},
    VOLUME = {52},
      YEAR = {2015},
    NUMBER = {1-2},
     PAGES = {327--344},
      ISSN = {0944-2669,1432-0835},
   MRCLASS = {32W20 (32Q15 32Q25 35J96)},
  MRNUMBER = {3299184},
MRREVIEWER = {Eleonora\ Di Nezza},
       DOI = {10.1007/s00526-014-0714-0},
       URL = {https://doi.org/10.1007/s00526-014-0714-0},
}

@article {FWW10,
    AUTHOR = {Fu, Jixiang and Wang, Zhizhang and Wu, Damin},
     TITLE = {Form-type {C}alabi-{Y}au equations},
   JOURNAL = {Math. Res. Lett.},
  FJOURNAL = {Mathematical Research Letters},
    VOLUME = {17},
      YEAR = {2010},
    NUMBER = {5},
     PAGES = {887--903},
      ISSN = {1073-2780},
   MRCLASS = {32W20 (32Q15 32Q25)},
  MRNUMBER = {2727616},
MRREVIEWER = {Valentino\ Tosatti},
       DOI = {10.4310/MRL.2010.v17.n5.a7},
       URL = {https://doi.org/10.4310/MRL.2010.v17.n5.a7},
}

@article {Gauduchon84,
    AUTHOR = {Gauduchon, Paul},
     TITLE = {La {$1$}-forme de torsion d'une vari\'{e}t\'{e} hermitienne
              compacte},
   JOURNAL = {Math. Ann.},
  FJOURNAL = {Mathematische Annalen},
    VOLUME = {267},
      YEAR = {1984},
    NUMBER = {4},
     PAGES = {495--518},
      ISSN = {0025-5831,1432-1807},
   MRCLASS = {53C55 (32C10 58E11)},
  MRNUMBER = {742896},
MRREVIEWER = {I.\ Vaisman},
       DOI = {10.1007/BF01455968},
       URL = {https://doi.org/10.1007/BF01455968},
}

@article {GGQ22,
    AUTHOR = {George, Mathew and Guan, Bo and Qiu, Chunhui},
     TITLE = {Fully nonlinear elliptic equations on {H}ermitian manifolds
              for symmetric functions of partial {L}aplacians},
   JOURNAL = {J. Geom. Anal.},
  FJOURNAL = {Journal of Geometric Analysis},
    VOLUME = {32},
      YEAR = {2022},
    NUMBER = {6},
     PAGES = {Paper No. 183, 27},
      ISSN = {1050-6926,1559-002X},
   MRCLASS = {35J15 (35J60 58J05)},
  MRNUMBER = {4411746},
MRREVIEWER = {Jari\ Taskinen},
       DOI = {10.1007/s12220-022-00918-y},
       URL = {https://doi.org/10.1007/s12220-022-00918-y},
}

@book {GT83,
    AUTHOR = {Gilbarg, David and Trudinger, Neil S.},
     TITLE = {Elliptic partial differential equations of second order},
    SERIES = {Grundlehren der mathematischen Wissenschaften [Fundamental
              Principles of Mathematical Sciences]},
    VOLUME = {224},
   EDITION = {Second},
 PUBLISHER = {Springer-Verlag, Berlin},
      YEAR = {1983},
     PAGES = {xiii+513},
      ISBN = {3-540-13025-X},
   MRCLASS = {35Jxx (35-01)},
  MRNUMBER = {737190},
MRREVIEWER = {O.\ John},
       DOI = {10.1007/978-3-642-61798-0},
       URL = {https://doi.org/10.1007/978-3-642-61798-0},
}

@article {Guan14,
    AUTHOR = {Guan, Bo},
     TITLE = {Second-order estimates and regularity for fully nonlinear
              elliptic equations on {R}iemannian manifolds},
   JOURNAL = {Duke Math. J.},
  FJOURNAL = {Duke Mathematical Journal},
    VOLUME = {163},
      YEAR = {2014},
    NUMBER = {8},
     PAGES = {1491--1524},
      ISSN = {0012-7094,1547-7398},
   MRCLASS = {35R01 (35B65 35J60 35J96 58J05)},
  MRNUMBER = {3284698},
MRREVIEWER = {Pierpaolo\ Esposito},
       DOI = {10.1215/00127094-2713591},
       URL = {https://doi.org/10.1215/00127094-2713591},
}

@article {GN23,
    AUTHOR = {Guan, Bo and Nie, Xiaolan},
     TITLE = {Fully nonlinear elliptic equations with gradient terms on
              {H}ermitian manifolds},
   JOURNAL = {Int. Math. Res. Not. IMRN},
  FJOURNAL = {International Mathematics Research Notices. IMRN},
      YEAR = {2023},
    NUMBER = {16},
     PAGES = {14006--14042},
      ISSN = {1073-7928,1687-0247},
   MRCLASS = {32W50 (35J65)},
  MRNUMBER = {4631427},
       DOI = {10.1093/imrn/rnac219},
       URL = {https://doi.org/10.1093/imrn/rnac219},
}

@article {GP23,
    AUTHOR = {Guo, Bin and Phong, Duong H.},
     TITLE = {Auxiliary {M}onge-{A}mp\`ere equations in geometric analysis},
   JOURNAL = {ICCM Not.},
  FJOURNAL = {ICCM Notices. Notices of the International Congress of Chinese
              Mathematicians},
    VOLUME = {11},
      YEAR = {2023},
    NUMBER = {1},
     PAGES = {98--135},
      ISSN = {2326-4810,2326-4845},
   MRCLASS = {35J96 (32W20 58J05)},
  MRNUMBER = {4614271},
}

@article {GPT23,
    AUTHOR = {Guo, Bin and Phong, Duong H. and Tong, Freid},
     TITLE = {On {$L^\infty$} estimates for complex {M}onge-{A}mp\`ere
              equations},
   JOURNAL = {Ann. of Math. (2)},
  FJOURNAL = {Annals of Mathematics. Second Series},
    VOLUME = {198},
      YEAR = {2023},
    NUMBER = {1},
     PAGES = {393--418},
      ISSN = {0003-486X,1939-8980},
   MRCLASS = {35J60 (32W20 35J96 53C55 53C56)},
  MRNUMBER = {4593734},
       DOI = {10.4007/annals.2023.198.1.4},
       URL = {https://doi.org/10.4007/annals.2023.198.1.4},
}

@article {GS19,
    AUTHOR = {Gursky, Matthew J. and Streets, Jeffrey},
     TITLE = {A formal {R}iemannian structure on conformal classes and the
              inverse {G}auss curvature flow},
   JOURNAL = {Geom. Flows},
  FJOURNAL = {Geometric Flows},
    VOLUME = {4},
      YEAR = {2019},
    NUMBER = {1},
     PAGES = {30--50},
      ISSN = {2353-3382},
   MRCLASS = {53E30 (58E11)},
  MRNUMBER = {4061505},
MRREVIEWER = {Julian\ Scheuer},
       DOI = {10.1515/geofl-2019-0003},
       URL = {https://doi.org/10.1515/geofl-2019-0003},
}

@article {He21,
    AUTHOR = {He, Weiyong},
     TITLE = {The {G}ursky-{S}treets equations},
   JOURNAL = {Math. Ann.},
  FJOURNAL = {Mathematische Annalen},
    VOLUME = {381},
      YEAR = {2021},
    NUMBER = {3-4},
     PAGES = {1085--1135},
      ISSN = {0025-5831,1432-1807},
   MRCLASS = {53C18},
  MRNUMBER = {4333410},
MRREVIEWER = {Yuxin\ Ge},
       DOI = {10.1007/s00208-020-02021-5},
       URL = {https://doi.org/10.1007/s00208-020-02021-5},
}

@article {LU17,
    AUTHOR = {Latorre, Adela and Ugarte, Luis},
     TITLE = {On non-{K}\"{a}hler compact complex manifolds with balanced
              and astheno-{K}\"{a}hler metrics},
   JOURNAL = {C. R. Math. Acad. Sci. Paris},
  FJOURNAL = {Comptes Rendus Math\'{e}matique. Acad\'{e}mie des Sciences.
              Paris},
    VOLUME = {355},
      YEAR = {2017},
    NUMBER = {1},
     PAGES = {90--93},
      ISSN = {1631-073X,1778-3569},
   MRCLASS = {32Q99 (32J27 53C55)},
  MRNUMBER = {3590290},
MRREVIEWER = {Daniele\ Angella},
       DOI = {10.1016/j.crma.2016.11.004},
       URL = {https://doi.org/10.1016/j.crma.2016.11.004},
}

@article {LY2005,
    AUTHOR = {Li, Jun and Yau, Shing-Tung},
     TITLE = {The existence of supersymmetric string theory with torsion},
   JOURNAL = {J. Differential Geom.},
  FJOURNAL = {Journal of Differential Geometry},
    VOLUME = {70},
      YEAR = {2005},
    NUMBER = {1},
     PAGES = {143--181},
      ISSN = {0022-040X,1945-743X},
   MRCLASS = {32Q25 (32G05 53C80 81T30 81T60)},
  MRNUMBER = {2192064},
MRREVIEWER = {Frederik\ Witt},
       URL = {http://projecteuclid.org/euclid.jdg/1143572017},
}

@article {Mabuchi87,
    AUTHOR = {Mabuchi, Toshiki},
     TITLE = {Some symplectic geometry on compact {K}\"{a}hler manifolds.
              {I}},
   JOURNAL = {Osaka J. Math.},
  FJOURNAL = {Osaka Journal of Mathematics},
    VOLUME = {24},
      YEAR = {1987},
    NUMBER = {2},
     PAGES = {227--252},
      ISSN = {0030-6126},
   MRCLASS = {53C55 (32C10 58B20 58D17 58E20)},
  MRNUMBER = {909015},
MRREVIEWER = {Yoshiko\ Kubo},
       URL = {http://projecteuclid.org/euclid.ojm/1200780161},
}

@article {Michelsohn82,
    AUTHOR = {Michelsohn, M. L.},
     TITLE = {On the existence of special metrics in complex geometry},
   JOURNAL = {Acta Math.},
  FJOURNAL = {Acta Mathematica},
    VOLUME = {149},
      YEAR = {1982},
    NUMBER = {3-4},
     PAGES = {261--295},
      ISSN = {0001-5962,1871-2509},
   MRCLASS = {53C55 (32C30)},
  MRNUMBER = {688351},
MRREVIEWER = {Steven\ M.\ Zucker},
       DOI = {10.1007/BF02392356},
       URL = {https://doi.org/10.1007/BF02392356},
}

@article {Semmes92,
    AUTHOR = {Semmes, Stephen},
     TITLE = {Complex {M}onge-{A}mp\`ere and symplectic manifolds},
   JOURNAL = {Amer. J. Math.},
  FJOURNAL = {American Journal of Mathematics},
    VOLUME = {114},
      YEAR = {1992},
    NUMBER = {3},
     PAGES = {495--550},
      ISSN = {0002-9327,1080-6377},
   MRCLASS = {32F07 (35J60 58F05)},
  MRNUMBER = {1165352},
       DOI = {10.2307/2374768},
       URL = {https://doi.org/10.2307/2374768},
}

@article {Strominger86,
    AUTHOR = {Strominger, Andrew},
     TITLE = {Superstrings with torsion},
   JOURNAL = {Nuclear Phys. B},
  FJOURNAL = {Nuclear Physics. B. Theoretical, Phenomenological, and
              Experimental High Energy Physics. Quantum Field Theory and
              Statistical Systems},
    VOLUME = {274},
      YEAR = {1986},
    NUMBER = {2},
     PAGES = {253--284},
      ISSN = {0550-3213,1873-1562},
   MRCLASS = {81E20 (81E13 81G20 83E15)},
  MRNUMBER = {851702},
MRREVIEWER = {Tsou Sheung Tsun},
       DOI = {10.1016/0550-3213(86)90286-5},
       URL = {https://doi.org/10.1016/0550-3213(86)90286-5},
}

@article {Szekelyhidi18,
    AUTHOR = {Sz\'{e}kelyhidi, G\'{a}bor},
     TITLE = {Fully non-linear elliptic equations on compact {H}ermitian
              manifolds},
   JOURNAL = {J. Differential Geom.},
  FJOURNAL = {Journal of Differential Geometry},
    VOLUME = {109},
      YEAR = {2018},
    NUMBER = {2},
     PAGES = {337--378},
      ISSN = {0022-040X,1945-743X},
   MRCLASS = {58J05 (32W50 35J60 35J96 35R01 53C55)},
  MRNUMBER = {3807322},
MRREVIEWER = {Bianca\ Santoro},
       DOI = {10.4310/jdg/1527040875},
       URL = {https://doi.org/10.4310/jdg/1527040875},
}

@article {STW17,
    AUTHOR = {Sz\'{e}kelyhidi, G\'{a}bor and Tosatti, Valentino and
              Weinkove, Ben},
     TITLE = {Gauduchon metrics with prescribed volume form},
   JOURNAL = {Acta Math.},
  FJOURNAL = {Acta Mathematica},
    VOLUME = {219},
      YEAR = {2017},
    NUMBER = {1},
     PAGES = {181--211},
      ISSN = {0001-5962,1871-2509},
   MRCLASS = {53C55 (32C36 32Q99)},
  MRNUMBER = {3765661},
MRREVIEWER = {Keizo\ Hasegawa},
       DOI = {10.4310/ACTA.2017.v219.n1.a6},
       URL = {https://doi.org/10.4310/ACTA.2017.v219.n1.a6},
}

@article {TW17,
    AUTHOR = {Tosatti, Valentino and Weinkove, Ben},
     TITLE = {The {M}onge-{A}mp\`ere equation for {$(n-1)$}-plurisubharmonic
              functions on a compact {K}\"{a}hler manifold},
   JOURNAL = {J. Amer. Math. Soc.},
  FJOURNAL = {Journal of the American Mathematical Society},
    VOLUME = {30},
      YEAR = {2017},
    NUMBER = {2},
     PAGES = {311--346},
      ISSN = {0894-0347,1088-6834},
   MRCLASS = {32W20 (32Q15 32U05 53C55)},
  MRNUMBER = {3600038},
MRREVIEWER = {Bianca\ Santoro},
       DOI = {10.1090/jams/875},
       URL = {https://doi.org/10.1090/jams/875},
}

@article {TW19,
    AUTHOR = {Tosatti, Valentino and Weinkove, Ben},
     TITLE = {Hermitian metrics, {$(n-1,n-1)$} forms and {M}onge-{A}mp\`ere
              equations},
   JOURNAL = {J. Reine Angew. Math.},
  FJOURNAL = {Journal f\"{u}r die Reine und Angewandte Mathematik. [Crelle's
              Journal]},
    VOLUME = {755},
      YEAR = {2019},
     PAGES = {67--101},
      ISSN = {0075-4102,1435-5345},
   MRCLASS = {32W20 (53C55)},
  MRNUMBER = {4015228},
MRREVIEWER = {Rafa\l \ Czy\.{z}},
       DOI = {10.1515/crelle-2017-0017},
       URL = {https://doi.org/10.1515/crelle-2017-0017},
}

@article {Yau78,
    AUTHOR = {Yau, Shing Tung},
     TITLE = {On the {R}icci curvature of a compact {K}\"{a}hler manifold
              and the complex {M}onge-{A}mp\`ere equation. {I}},
   JOURNAL = {Comm. Pure Appl. Math.},
  FJOURNAL = {Communications on Pure and Applied Mathematics},
    VOLUME = {31},
      YEAR = {1978},
    NUMBER = {3},
     PAGES = {339--411},
      ISSN = {0010-3640,1097-0312},
   MRCLASS = {53C55 (32C10 35J60)},
  MRNUMBER = {480350},
MRREVIEWER = {Robert\ E.\ Greene},
       DOI = {10.1002/cpa.3160310304},
       URL = {https://doi.org/10.1002/cpa.3160310304},
}

@misc{GP23-1,
      title={On $L^\infty$ estimates for fully nonlinear partial differential equations}, 
      author={Bin Guo and Duong H. Phong},
      year={2023},
      eprint={2204.12549},
      archivePrefix={arXiv},
      primaryClass={math.DG}
}

@article {CC21,
    AUTHOR = {Chen, Xiuxiong and Cheng, Jingrui},
     TITLE = {On the constant scalar curvature {K}\"{a}hler metrics
              ({I})---{A} priori estimates},
   JOURNAL = {J. Amer. Math. Soc.},
  FJOURNAL = {Journal of the American Mathematical Society},
    VOLUME = {34},
      YEAR = {2021},
    NUMBER = {4},
     PAGES = {909--936},
      ISSN = {0894-0347,1088-6834},
   MRCLASS = {53C55 (53C21)},
  MRNUMBER = {4301557},
MRREVIEWER = {Kai\ Zheng},
       DOI = {10.1090/jams/967},
       URL = {https://doi.org/10.1090/jams/967},
}

@misc{KLW23,
      title={On uniform estimates for $(n-1)-$form fully nonlinear partial differential equations on compact Hermitian manifolds}, 
      author={Nikita Klemyatin and Shuang Liang and Chuwen Wang},
      year={2023},
      eprint={2211.13798},
      archivePrefix={arXiv},
      primaryClass={math.AP}
}

@article {JY93,
    AUTHOR = {Jost, J\"{u}rgen and Yau, Shing-Tung},
     TITLE = {A nonlinear elliptic system for maps from {H}ermitian to
              {R}iemannian manifolds and rigidity theorems in {H}ermitian
              geometry},
   JOURNAL = {Acta Math.},
  FJOURNAL = {Acta Mathematica},
    VOLUME = {170},
      YEAR = {1993},
    NUMBER = {2},
     PAGES = {221--254},
      ISSN = {0001-5962,1871-2509},
   MRCLASS = {58E20 (53C55 57R55 58G30)},
  MRNUMBER = {1226528},
MRREVIEWER = {Guo\ Ying\ Jiang},
       DOI = {10.1007/BF02392786},
       URL = {https://doi.org/10.1007/BF02392786},
}

@article {CR23,
    AUTHOR = {Chiose, Ionu\c{t} and R\u{a}sdeaconu, Rare\c{s}},
     TITLE = {Remarks on astheno-{K}\"{a}hler manifolds, {B}ott-{C}hern and
              {A}eppli cohomology groups},
   JOURNAL = {Ann. Global Anal. Geom.},
  FJOURNAL = {Annals of Global Analysis and Geometry},
    VOLUME = {63},
      YEAR = {2023},
    NUMBER = {3},
     PAGES = {Paper No. 24, 23},
      ISSN = {0232-704X,1572-9060},
   MRCLASS = {53C55 (32J18)},
  MRNUMBER = {4581543},
MRREVIEWER = {Quanting\ Zhao},
       DOI = {10.1007/s10455-023-09903-2},
       URL = {https://doi.org/10.1007/s10455-023-09903-2},
}

@article {FLY12,
    AUTHOR = {Fu, Jixiang and Li, Jun and Yau, Shing-Tung},
     TITLE = {Balanced metrics on non-{K}\"{a}hler {C}alabi-{Y}au
              threefolds},
   JOURNAL = {J. Differential Geom.},
  FJOURNAL = {Journal of Differential Geometry},
    VOLUME = {90},
      YEAR = {2012},
    NUMBER = {1},
     PAGES = {81--129},
      ISSN = {0022-040X,1945-743X},
   MRCLASS = {32Q25 (53C55)},
  MRNUMBER = {2891478},
MRREVIEWER = {Yuguang\ Zhang},
       URL = {http://projecteuclid.org/euclid.jdg/1335209490},
}

@article {TW10,
    AUTHOR = {Tosatti, Valentino and Weinkove, Ben},
     TITLE = {Estimates for the complex {M}onge-{A}mp\`ere equation on
              {H}ermitian and balanced manifolds},
   JOURNAL = {Asian J. Math.},
  FJOURNAL = {Asian Journal of Mathematics},
    VOLUME = {14},
      YEAR = {2010},
    NUMBER = {1},
     PAGES = {19--40},
      ISSN = {1093-6106,1945-0036},
   MRCLASS = {32W20 (32Q25 35J96 35R01)},
  MRNUMBER = {2726593},
MRREVIEWER = {Yanir\ A.\ Rubinstein},
       DOI = {10.4310/AJM.2010.v14.n1.a3},
       URL = {https://doi.org/10.4310/AJM.2010.v14.n1.a3},
}

\end{document}